\documentclass[a4paper,11pt]{amsart}

\usepackage[utf8]{inputenc}

\usepackage{amssymb}

\usepackage[T1]{fontenc}

\allowdisplaybreaks[2]

\providecommand{\norm}[1]{\lVert#1\rVert}
\providecommand{\tnorma}[1]{\lvert\lvert\lvert#1\rvert\rvert\rvert}

\DeclareMathOperator{\trace}{trace}

\newtheorem{theorem}{Theorem}[section]
\newtheorem{proposition}[theorem]{Proposition}
\newtheorem{lemma}[theorem]{Lemma}
\newtheorem{corollary}[theorem]{Corollary}
\theoremstyle{definition}
\newtheorem{remark}[theorem]{Remark}

\newtheorem{example}[theorem]{Example}

\numberwithin{equation}{section}

\begin{document}

\title[Orthogonally additive polynomials]
{Orthogonally additive polynomials on convolution algebras associated with a compact group}

\author{J. Alaminos}
\address{Departamento de An\' alisis
Matem\' atico\\ Fa\-cul\-tad de Ciencias\\ Universidad de Granada\\
18071 Granada, Spain} 
\email{alaminos@ugr.es}
\author{M. L. C. Godoy}
\address{Departamento de An\' alisis
Matem\' atico\\ Fa\-cul\-tad de Ciencias\\ Universidad de Granada\\
18071 Granada, Spain} 
\email{marisa23@correo.ugr.es}
\author{J. Extremera}
\address{Departamento de An\' alisis
Matem\' atico\\ Fa\-cul\-tad de Ciencias\\ Universidad de Granada\\
18071 Granada, Spain}
\email{jlizana@ugr.es}
\author{A.\,R. Villena}
\address{Departamento de An\' alisis
Matem\' atico\\ Fa\-cul\-tad de
 Ciencias\\ Universidad de Granada\\
18071 Granada, Spain} 
\email{avillena@ugr.es}

\date{}

\begin{abstract}
Let $G$ be a compact group,
let $X$ be a Banach space, and
let $P\colon L^1(G)\to X$ be an orthogonally additive, continuous $n$-homogeneous polynomial.
Then we show that there exists a unique continuous linear map $\Phi\colon L^1(G)\to X$ such that
$P(f)=\Phi \bigl(f\ast\stackrel{n}{\cdots}\ast f \bigr)$ for each $f\in L^1(G)$.
We also seek analogues of this result about $L^1(G)$ for various other convolution algebras,
including $L^p(G)$, for $1< p\le\infty$, and $C(G)$.
\end{abstract}


\thanks{The first, the third and the fourth named authors were supported by MINECO grant MTM2015--65020--P and Junta de Andaluc\'{\i}a grant FQM--185.}

\maketitle

\section{Introduction}

Throughout all  algebras and linear spaces are complex.
Of course, linearity is understood to mean complex linearity.
Moreover, we fix $n\in\mathbb{N}$ with $n\ge 2$.

Let $X$ and $Y$ be linear spaces. 
A map $P\colon X\to Y$ is said to be an \emph{$n$-homogeneous polynomial} if
there exists an $n$-linear map $\varphi\colon X^n\to Y$ such that
$P(x)=\varphi \left( x,\dotsc,x \right)$ $(x\in X)$. Here and subsequently, $X^n$
stands for the $n$-fold Cartesian product of $X$. Such a map is unique
if it is required to be symmetric. This is a consequence of the
so-called polarization formula which defines $\varphi$ by
\[
\varphi \left( x_1,\ldots,x_n \right)=
\frac{1}{n!\,2^n}\sum_{\epsilon_1,\ldots,\epsilon_n=\pm 1}
\epsilon_{1}\cdots\epsilon_{n} P \left(\epsilon_1x_1+\cdots+\epsilon_{n}x_{n} \right)
\]
for all $x_1,\dotsc,x_n\in X$.
Further, in the case where $X$ and $Y$ are normed spaces, the
polynomial $P$ is continuous if and only if the symmetric $n$-linear
map $\varphi$ associated with $P$ is continuous.
Let $A$ be an algebra. 
Then the map $P_n\colon A\to A$ defined by
\begin{equation*}
P_n(a)=a^n \quad (a\in A)
\end{equation*}
is a prototypical example of $n$-homogeneous polynomial. 
The symmetric $n$-linear map associated with $P_n$ is the map 
$S_n\colon A^n\to A$ defined by
\begin{equation*}
S_n \left( a_1,\dotsc,a_n \right) =
\frac{1}{n!}\sum_{\sigma\in\mathfrak{S}_n}a_{\sigma(1)}\dotsb a_{\sigma(n)} 
\quad \left(a_1,\dots,a_n\in A \right),
\end{equation*}
where $\mathfrak{S}_n$ stands for the symmetric group of order $n$.
From now on, 
we write $\mathcal{P}_n(A)$ for the linear span of the set $\left\{a^n : a\in A\right\}$.
Given a linear space $Y$ and a linear map $\Phi\colon\mathcal{P}_n(A)\to Y$,
the map $P\colon A\to Y$ defined by
\begin{equation}\label{standard0}
P(a)=\Phi\left( a^n \right)  \quad (a\in A)
\end{equation}
yields a particularly important example of $n$-homogeneous polynomial, and
one might wish to know an algebraic characterization of those  
$n$-homogeneous polynomials $P\colon A\to Y$ which can be expressed in the form \eqref{standard0}.
Further, in the case where $A$ is a Banach algebra, $Y$ is a Banach space, and 
the $n$-homogeneous polynomial $P\colon A\to Y$ is continuous, one should
particularly like that the map $\Phi$ of \eqref{standard0} be continuous.
A property that has proven valuable for this purpose is the so-called orthogonal additivity.
Let $A$ be an algebra and let $Y$ be a linear space.
A map $P\colon A\to Y$ is said to be \emph{orthogonally additive}
if
\[
a,b\in A, \ ab=ba=0  \ \Rightarrow \ P(a+b)=P(a)+P(b).
\]
The polynomial defined by \eqref{standard0} is a prototypical example of orthogonally
additive $n$-homogeneous polynomial, and 
the obvious questions that one can address are the following.
\begin{enumerate}
\item[Q1]
Let $A$ be a specified algebra.
Is it true that every orthogonally additive $n$-homogeneous polynomial $P$ from $A$
into each linear space $Y$ can be expressed in the standard form \eqref{standard0} for some
linear map $\Phi\colon\mathcal{P}_n(A)\to Y$?
\item[Q2]
Let $A$ be a specified Banach algebra.
Is it true that every  orthogonally additive continuous $n$-homogeneous polynomial $P$ from $A$
into each Banach space $Y$ can be expressed in the standard form \eqref{standard0} for some
continuous linear map $\Phi\colon\mathcal{P}_n(A)\to Y$?
\item[Q3]
Let $A$ be a specified Banach algebra.
Is there any norm $\tnorma{\cdot}$ on $\mathcal{P}_n(A)$ with
the property that the  orthogonally additive continuous $n$-homogeneous polynomials
from $A$ into each Banach space $Y$ are exactly the polynomials of the form \eqref{standard0}
for some
$\tnorma{\cdot}$-continuous linear map $\Phi\colon\mathcal{P}_n(A)\to Y$?
\end{enumerate}
It seems to be convenient to remark that the demand of Q3 results precisely in the following two conditions:
\begin{itemize}
\item
for each Banach space $Y$ and each $\tnorma{\cdot}$-continuous linear map 
$\Phi\colon\mathcal{P}_n(A)\to Y$, the prototypical polinomial $P\colon A\to Y$ defined by 
\eqref{standard0} is continuous, and
\item
every  orthogonally additive continuous $n$-homogeneous polynomial $P$ from $A$
into each Banach space $Y$ can be expressed in the standard form~\eqref{standard0} for some
$\tnorma{\cdot}$-continuous linear map $\Phi\colon\mathcal{P}_n(A)\to Y$.
\end{itemize}
It is shown in \cite{P} that the answer to Question Q2 is positive 
in the case where $A$ is a $C^*$-algebra 
(see~\cite{P2,P3} for the case where $A$ is a $C^*$-algebra and $P$ is
a holomorphic map).
The references \cite{A1,A2,V,W,WW} discuss Question Q2 for a variety of Banach function algebras,
including the Fourier algebra $A(G)$ and the Fig\`a-Talamanca-Herz algebra $A_p(G)$
of a locally compact group $G$. 

This paper focuses on the questions Q1, Q2, and Q3 mentioned above for a variety of convolution algebras 
associated with a compact group $G$, such as $L^p(G)$, for $1\le p\le\infty$, and $C(G)$.
In contrast to the previous references, that are concerned with $C^*$-algebras and commutative Banach algebras,
the algebras in this work are neither $C^*$ nor commutative. 

Throughout, we are concerned with a compact group $G$ whose Haar measure is normalized. 
We write $\int_G f(t)\, dt$ for the integral of $f\in L^1(G)$ with respect to the Haar measure.
For $f\in L^1(G)$, we denote by $f^{*n}$ the $n$-fold convolution product $f\ast\dotsb\ast f$.
We denote by $\widehat{G}$ the set of equivalence classes of irreducible unitary representations
of $G$.
Let $\pi$ be an irreducible unitary representation of $G$ on a Hilbert space $H_\pi$.
We set $d_\pi=\dim(H_\pi)(<\infty)$, and the character $\chi_\pi$ of $\pi$ is the
continuous function on $G$ defined by
\[
\chi_\pi(t)=\trace\bigl(\pi(t)\bigr) \quad (t\in G).
\]
We write $\mathcal{T}_\pi(G)$ for the linear subspace of $C(G)$ generated by the set of continuous functions on
$G$ of the form $t\mapsto\langle\pi(t)u\vert v\rangle$ as $u$ and $v$ range over $H_\pi$.
It should be pointed out that $\chi_\pi$ and $\mathcal{T}_\pi(G)$  depend only on the unitary equivalence class of $\pi$.
We write $\mathcal{T}(G)$ for the linear span of the functions in $\mathcal{T}_\pi(G)$ as $[\pi]$
ranges over $\widehat{G}$. 
Then $\mathcal{T}(G)$ is a two-sided ideal of $L^1(G)$ whose elements are called trigonometric polynomials on $G$.
The Fourier transform of a function $f\in L^1(G)$ at $\pi$ is defined to be the operator
\[
\widehat{f}(\pi)=\int_Gf(t)\pi(t^{-1}) \,dt
\]
on $H_\pi$. Note that if $\pi'$ is equivalent to $\pi$, then the operators $\widehat{f}(\pi')$
and $\widehat{f}(\pi)$ are unitarily equivalent.

In Section~2 we show that the answer to Question Q1 is positive for the algebra $\mathcal{T}(G)$.
In Section~3 we show that the answer to Question Q2 is positive for the group algebra $L^1(G)$.
In Section~4 we give a negative answer to Question Q2 for any of the convolution algebras $L^p(\mathbb{T})$,
for $1<p\le\infty$, and $C(\mathbb{T})$, where $\mathbb{T}$ denotes the circle group.
In Section~5 we prove that, for each Banach algebra $A$, there exists a largest norm topology on the linear space
$\mathcal{P}_n(A)$ for which the answer to Question Q3 can be positive.
Finally,  in Section~6 we show that the answer to Question Q3 is positive for most of the significant convolution
algebras associated to $G$, such as $L^p(G)$, for $1<p<\infty$, and $C(G)$, when considering the norm introduced in Section~5.

We presume a basic knowledge of Banach algebra theory, harmonic analysis for compact groups,  and polynomials on Banach spaces.
For the relevant background material concerning these topics, see \cite{D}, \cite{HR},  and \cite{M}, respectively.

\section{Orthogonally additive polynomials on $\mathcal{T}(G)$}

Our starting point is furnished by applying \cite{P} to the full matrix algebra $\mathbb{M}_k$
of order $k$
(which supplies the most elementary example of $C^*$-algebra).

\begin{lemma}\label{l1}
Let $\mathcal{M}$ be an algebra isomorphic to $\mathbb{M}_k$ for some $k\in\mathbb{N}$,
let $X$ be a linear space, and 
let $P\colon\mathcal{M}\to X$ be an orthogonally additive $n$-homogeneous polynomial.
Then there exists a unique linear map $\Phi\colon\mathcal{M}\to X$ such that
$P(a)=\Phi\left( a^n \right) $ for each $a\in\mathcal{M}$.
Further, if $\varphi\colon\mathcal{M}^n\to X$ is the symmetric $n$-linear map associated with $P$ 
and $e$ is the identity of $\mathcal{M}$, then
$\Phi(a)=\varphi \left(a,e,\dotsc,e \right)$
for each $a\in \mathcal{M}$.
\end{lemma}

\begin{proof}
Let $\Psi\colon\mathcal{M}\to\mathbb{M}_k$ be an isomorphism.
Endow $X$ with a norm, and let $Y$ be its completion.
Since $\mathbb{M}_k$ is a $C^*$-algebra and 
the map $P\circ\Psi^{-1}\colon\mathbb{M}_k\to Y$ is a continuous orthogonally additive
$n$-homogeneous polynomial, \cite[Corollary~3.1]{P} then shows that there exists
a unique linear map $\Theta\colon\mathbb{M}_k\to Y$ such that
$P \left(\Psi^{-1}(M)\right)=\Theta \left(M^n \right)$ for each $M\in\mathbb{M}_k$. It is a simple matter to check
that the map $\Phi=\Theta\circ\Psi$ satisfies the identity $P(a)=\Phi \left(a^n \right)$ $(a\in A)$.
Now the polarization of this identity yields 
$\varphi \left( a_1,\ldots,a_n \right)=\Phi\bigl(S_n(a_1,\dotsc,a_n)\bigr)$
$\left( a_1,\dotsc,a_n\in \mathcal{M}\right)$,
whence
$\varphi \left( a,e,\dotsc,e \right)=\Phi(a)$ for each $a\in\mathcal{M}$.
\end{proof}

In what follows, we will require some elementary facts about the algebra $\mathcal{T}(G)$;
we gather together these facts here for reference.

\begin{lemma}\label{l2}
Let $G$ be a compact group.
Then the following results hold.
\begin{enumerate}
\item
For each irreducible unitary representation $\pi$,
$\mathcal{T}_{\pi}(G)$ is a minimal two-sided ideal of $L^1(G)$, 
$\mathcal{T}_{\pi}(G)$ is isomorphic to the full matrix algebra $\mathbb{M}_{d_\pi}$, and
$d_\pi\chi_\pi$ is the identity of $\mathcal{T}_\pi(G)$.
\item
For each $f\in\mathcal{T}(G)$, the set 
\[
\bigl\{[\pi]\in\widehat{G} : f\ast\chi_\pi\ne 0\bigr\}
\]
is finite and
\[
f=\sum_{[\pi]\in\widehat{G}}d_\pi f\ast\chi_\pi.
\]
\end{enumerate}
\end{lemma}

\begin{proof}
\begin{enumerate}
\item \cite[Theorems~27.21 and 27.24(ii)]{HR}.
\item \cite[Remark~27.8(a) and Theorem~27.24(ii)]{HR}.   \qedhere
\end{enumerate}
\end{proof}

\begin{theorem}\label{t1}
Let $G$ be a compact group, 
let $X$ be a linear space,  and
let $P\colon \mathcal{T}(G)\to X$ be an orthogonally additive $n$-homogeneous polynomial.
Then there exists a unique linear map $\Phi\colon\mathcal{T}(G)\to X$
such that
\begin{equation*}
P(f)=\Phi \left( f^{*n} \right)
\end{equation*}
for each $f\in\mathcal{T}(G)$.
Further,
if $\varphi\colon\mathcal{T}(G)^{n}\to X$ is the symmetric $n$-linear map associated with $P$, 
then 
\begin{equation*}
\Phi(f)=
\sum_{[\pi]\in\widehat{G}}\varphi \left( d_\pi f\ast\chi_\pi,d_\pi\chi_\pi\dotsc,d_\pi \chi_\pi \right)
\end{equation*}
(the term $f\ast\chi_\pi$ being $0$ for all but finitely many $[\pi]\in\widehat{G}$)
for each $f\in\mathcal{T}(G)$.
\end{theorem} 

\begin{proof}
We first show that
\begin{equation}\label{e1}
P(f)=\sum_{[\pi]\in\widehat{G}}P \left( d_\pi f\ast\chi_\pi \right)
\end{equation}
for each $f\in \mathcal{T}(G)$.
Set $f\in\mathcal{T}(G)$.
By Lemma~\ref{l2}, we have
\begin{equation}\label{e2}
f=\sum_{[\pi]\in\widehat{G}}d_\pi f\ast\chi_\pi.
\end{equation}
It should be pointed out that all save finitely many of the summands in \eqref{e2} are $0$.
Further, from \cite[Theorem~27.24(ii)-(iii)]{HR} we see that
\[
\left(f\ast\chi_\pi \right)\ast \left( f\ast\chi_{\pi'} \right)=
f\ast f\ast\chi_\pi\ast\chi_{\pi'}=0,
\]
whenever $[\pi],[\pi']\in\widehat{G}$ and  $[\pi]\ne[\pi']$.
The orthogonal additivity of $P$ and \eqref{e2} then yield \eqref{e1}.

Now set $[\pi]\in\widehat{G}$.
On account of Lemma~\ref{l2}(1),
$\mathcal{T}_{\pi}(G)$ is isomorphic to the full matrix algebra $\mathbb{M}_{d_\pi}$
and $d_\pi\chi_\pi$ is the identity of $\mathcal{T}_{\pi}(G)$.
Hence, by Lemma~\ref{l1}, we have
\begin{equation*}
P(f)=\varphi \left(f^{*n},d_\pi\chi_\pi,\dotsc,d_\pi\chi_\pi \right)
\end{equation*}
for each $f\in\mathcal{T}_{\pi}(G)$. 
In particular, if $f\in\mathcal{T}(G)$, then
$f\ast d_\pi\chi_\pi\in\mathcal{T}_{\pi}(G)$ and thus
\begin{equation}\label{e3}
\begin{split}
P \left(d_\pi f\ast\chi_\pi \right)
& =
\varphi\bigl((f\ast d_\pi\chi_\pi)^{*n},d_\pi\chi_\pi,\dotsc,d_\pi\chi_\pi\bigr)\\
& =
\varphi\bigl(f^{*n}\ast d_\pi\chi_\pi,d_\pi\chi_\pi,\dotsc,d_\pi\chi_\pi\bigr).
\end{split}
\end{equation}

From \eqref{e1} and \eqref{e3} we conclude that
\begin{equation*}
\begin{split}
P(f)
& =
\sum_{[\pi]\in\widehat{G}}P \left( d_\pi f\ast\chi_\pi \right)\\
& =
\sum_{[\pi]\in\widehat{G}}\varphi \left( d_\pi f^{*n}\ast\chi_\pi, d_\pi\chi_\pi,\dotsc,d_\pi\chi_\pi \right)
\end{split}
\end{equation*}
for each $f\in\mathcal{T}(G)$, which completes the proof.
\end{proof}

\section{Orthogonally additive polynomials on $L^1(G)$}\label{sect}

A key fact in what follows is that the Banach algebra $L^1(G)$ has a central bounded approximate identity
consisting of trigonometric polynomials.
Indeed, by \cite[Theorem~28.53]{HR},
there exists a bounded approximate identity $(h_\lambda)_{\lambda\in\Lambda}$ for $L^1(G)$ such that:
\begin{itemize}
\item
$\lVert h_\lambda\rVert_1=1$ for each $\lambda\in\Lambda$;
\item
$h_\lambda\ast f=f\ast h_\lambda$ for all $f\in L^1(G)$ and $\lambda\in\Lambda$;
\item
$\widehat{h_\lambda}(\pi)=\alpha_\lambda(\pi) I_\pi$ for some $\alpha_\lambda(\pi)\in\mathbb{R}^+_0$
for all $[\pi]\in\widehat{G}$ and $\lambda\in\Lambda$ (here $I_\pi$ denotes the identity operator on $H_\pi$);
\item
$\lim_{\lambda\in\Lambda} \alpha_\lambda(\pi)=1$ for each $[\pi]\in\widehat{G}$.
\end{itemize}
This approximate identity  will be used repeatedly hereafter.

\begin{theorem}\label{tm}
Let $G$ be a compact group,
let $X$ be a Banach space, and
let $P\colon L^1(G)\to X$ be a continuous $n$-homogeneous polynomial.
Then the following conditions are equivalent:
\begin{enumerate}
\item
the polynomial $P$ is orthogonally additive; 
\item
the polynomial $P$ is orthogonally additive on $\mathcal{T}(G)$, i.e.,
$P(f+g)=P(f)+P(g)$ whenever $f,g\in\mathcal{T}(G)$ are such that $f\ast g=g\ast f=0$;
\item
there exists a unique continuous linear map $\Phi\colon L^1(G)\to X$ such that
$P(f)=\Phi\left( f^{*n} \right) $ for each $f\in L^1(G)$.
\end{enumerate}
\end{theorem}

\begin{proof}
It is clear that $(1)\Rightarrow(2)$ and that $(3)\Rightarrow(1)$.
We will henceforth prove that $(2)\Rightarrow(3)$.

Let $\varphi\colon L^1(G)^{n}\to X$ be the symmetric $n$-linear map associated with $P$,
and let $\Phi_0\colon\mathcal{T}(G)\to X$ be the linear map defined by  
\begin{equation*}
\Phi_0(f)=
\sum_{[\pi]\in\widehat{G}}\varphi \left(d_\pi f\ast\chi_\pi,d_\pi\chi_\pi\dotsc,d_\pi \chi_\pi \right)
\end{equation*}
for each $f\in\mathcal{T}(G)$.
Since $P$ is orthogonally additive on $\mathcal{T}(G)$, Theorem~\ref{t1} yields
\begin{equation}\label{e12}
P(f)=\Phi_0 \left( f^{*n} \right) \quad \left(f\in\mathcal{T}(G) \right).
\end{equation}

We claim that $\Phi_0$ is continuous.
Let $(h_\lambda)_{\lambda\in\Lambda}$ be as introduced in the beginning of this section.
We now note that, for $\lambda\in\Lambda$ and $[\pi]\in\widehat{G}$,
\begin{equation*}
\begin{split}
\bigl(h_\lambda\ast\chi_\pi\bigr)(t)
& =
\int_G h_\lambda(s)\trace \bigl(\pi(s^{-1}t)\bigr) \, ds\\
& =
\trace \left(\int_G h_\lambda(s)\pi(s^{-1}t) \, ds\right)\\
& =
\trace \left(\int_G h_\lambda(s)\pi(s^{-1})\pi(t) \, ds\right)\\
& =
\trace \left(\left(\int_G h_\lambda(s)\pi(s^{-1}) \, ds\right)\pi(t)\right)\\
& =
\trace \bigl(\widehat{h_\lambda}(\pi)\pi(t)\bigr)\\
& =
\trace \bigl(\alpha_\lambda(\pi)\pi(t)\bigr)\\
& =
\alpha_\lambda(\pi)\trace \bigl(\pi(t)\bigr)\\
& =
\alpha_\lambda(\pi)\chi_\pi(t).
\end{split}
\end{equation*}
From Theorem~\ref{t1} and the polarization formula we deduce that
\begin{equation*}
\begin{split}
\varphi \left( f_1,\dotsc,f_n \right)
& =
\frac{1}{n!}\sum_{\sigma\in\mathfrak{S}_n} \Phi_0 \left( f_{\sigma(1)}\ast\dotsb\ast f_{\sigma(n)} \right)\\
& =
\frac{1}{n!}\sum_{[\pi]\in\widehat{G}}\sum_{\sigma\in\mathfrak{S}_n}
\varphi\bigl(
d_\pi f_{\sigma(1)}\ast\dotsb\ast f_{\sigma(n)}\ast\chi_\pi,d_\pi\chi_\pi,\dotsc,d_\pi\chi_\pi
\bigr)
\end{split}
\end{equation*}
for all $f_1,\dotsc,f_n\in\mathcal{T}(G)$.
Pick $f\in\mathcal{T}(G)$, and set
$\mathcal{F}=\bigl\{[\pi] : f\ast\chi_\pi\ne 0\bigr\}$.
We apply the above equation in the case where $f_1=f$ and $f_2=\dotsb=f_n=h_\lambda$ with $\lambda\in\Lambda$.
Since
\[
f_{\sigma(1)}\ast\dotsb \ast f_{\sigma(n)}\ast\chi_\pi=
f\ast h_\lambda\ast\dotsb\ast h_\lambda\ast\chi_\pi=
\alpha_\lambda(\pi)^{n-1}f\ast\chi_\pi
\]
for each $\sigma\in\mathfrak{S}_n$,
it follows that
\begin{equation*}
\begin{split}
\varphi \left( f,h_\lambda,\dotsc,h_\lambda \right)
& =
\sum_{[\pi]\in\widehat{G}}
\alpha_\lambda(\pi)^{n-1} \varphi \left( d_\pi f\ast\chi_\pi,d_\pi \chi_\pi,\dotsc,d_\pi \chi_\pi \right)\\
& =
\sum_{[\pi]\in\mathcal{F}}
\alpha_\lambda(\pi)^{n-1} \varphi \left( d_\pi f\ast\chi_\pi,d_\pi \chi_\pi,\dotsc,d_\pi \chi_\pi \right).
\end{split}
\end{equation*}
Since $\mathcal{F}$ is finite (Lemma~\ref{l2}(2)) and  
$\lim_{\lambda\in\Lambda}\alpha_\lambda(\pi)=1$ for each $[\pi]\in\widehat{G}$, we see that
the net $\bigl(\varphi \left(f,h_\lambda,\dotsc,h_\lambda \right)\bigr)_{\lambda\in\Lambda}$ is convergent and 
\begin{equation}\label{e13}
\lim_{\lambda\in\Lambda}\varphi \left( f,h_\lambda,\dotsc,h_\lambda \right)=
\sum_{[\pi]\in\mathcal{F}}\varphi \left( d_\pi f\ast\chi_\pi,d_\pi \chi_\pi,\dotsc,d_\pi \chi_\pi \right)=
\Phi_0(f).
\end{equation}
On the other hand, for each $\lambda\in\Lambda$, we have
\begin{equation}\label{e14}
\left\Vert \varphi \left( f,h_\lambda,\dotsc,h_\lambda \right)\right\Vert
\le
\left\Vert \varphi \right\Vert \left\Vert h_\lambda\right\Vert^{n-1} 
\left\Vert f\right\Vert \le \left\Vert \varphi \right\Vert \left\Vert f \right\Vert.
\end{equation}
By \eqref{e13} and \eqref{e14}, 
\[
\left\Vert \Phi_0(f) \right\Vert \le \left\Vert \varphi \right\Vert
\left\Vert f \right\Vert,
\]
which gives the continuity of $\Phi_0$.

Since $\mathcal{T}(G)$ is dense in $L^1(G)$ and $\Phi_0$ is continuous, 
there exists a unique continuous linear map $\Phi\colon L^1(G)\to X$ which 
extends $\Phi_0$. 
Since both $P$ and $\Phi$ are continuous, \eqref{e12} gives
$P(f)=\Phi \left(f^{*n} \right)$ for each $f\in L^1(G)$.

Our final task is to prove the uniqueness of the map $\Phi$.
Suppose that $\Psi\colon L^1(G)\to X$ is a continuous linear map
such that $P(f)=\Psi\left( f^{*n} \right) $ for each $f\in L^1(G)$.
By Theorem~\ref{t1}, $\Psi(f)=\Phi(f)\bigl(=\Phi_0(f)\bigr)$
for each $f\in\mathcal{T}(G)$.
Since $\mathcal{T}(G)$ is dense in $L^1(G)$, and both
$\Phi$ and $\Psi$ are continuous, it follows
that $\Psi(f)=\Phi(f)$ for each $f\in L^1(G)$.
\end{proof}

\section{Orthogonally additive polynomials on the convolution algebras $L^p(\mathbb{T})$, $1<p\le\infty$, and $C(\mathbb{T})$}

The next examples show that, if
$A$ is any of the convolution algebras $L^p(\mathbb{T})$, for $1< p\le\infty$, or $C(\mathbb{T})$,
then there exists an orthogonally additive, continuous $2$-homogeneous polynomial $P\colon A\to\mathbb{C}$
which cannot be expressed in the form $P(f)=\Phi \left(f\ast f\right)$ $(f\in A)$ for any continuous linear functional 
$\Phi\colon A\to\mathbb{C}$. Throughout this section,
$\mathbb{T}$ denotes the circle group $\{z\in\mathbb{C} : \vert z\vert=1\}$, and,
for $f\in L^1(\mathbb{T})$ and $k\in\mathbb{Z}$, $\widehat{f}(k)$ denotes the $k$th Fourier
coefficient of $f$.
For each $k\in\mathbb{Z}$, let $\chi_k\colon\mathbb{T}\to\mathbb{C}$ be the function defined by
\[
\chi_k(z)=z^k \quad (z\in\mathbb{T}).
\]  
Then
\begin{equation}\label{ca00}
\chi_k\ast \chi_k=\chi_k \quad (k\in\mathbb{Z})
\end{equation}
and
\begin{equation}\label{ca0}
\widehat{\chi_k}(j)=\delta_{jk} \quad (j,k\in\mathbb{Z}).
\end{equation}

\begin{example}\label{ca11}
Assume that $1<p<2$.
Set $q=\frac{p}{p-1}$, $r=\frac{p}{2-p}$, and $s=\frac{q}{2}$,
so that $\tfrac{1}{p}+\tfrac{1}{q}=1$ and $\tfrac{1}{r}+\tfrac{1}{s}=1$.
Take $h\in L^p(\mathbb{T})$ such that 
\begin{equation}\label{ca1}
\sum_{k=-\infty}^{+\infty}\vert\widehat{h}(k)\vert^s=+\infty.
\end{equation}
Such a choice is possible because of \cite[13.5.3(1)]{E}, since $s<q$.
We claim that there exists $a\in\ell^r(\mathbb{Z})$ such that
the sequence $\bigl(\sum_{k=-m}^m a(k)\widehat{h}(k)\bigr)$ does not converge.
To see this, we define the sequence $(\phi_m)$ in the dual of $\ell^r(\mathbb{Z})$ by
\[
\phi_m(a)=\sum_{k=-m}^m a(k)\widehat{h}(k) \quad (a\in\ell^r(\mathbb{Z}), \ m\in\mathbb{N}).
\]
It is immediate to check that
\[
\Vert\phi_m\Vert=\left(\sum_{k=-m}^m\bigl\vert\widehat{h}(k)\bigr\vert^s\right)^{1/s} \quad (m\in\mathbb{N}).
\]
From \eqref{ca1} we deduce that $(\lVert\phi_m\rVert)$ is unbounded, and the Banach-Steinhaus theorem then shows that
there exists $a\in\ell^r(\mathbb{Z})$ such that $(\phi_m(a))$ does not converge, as claimed.

Let $f\in L^p(\mathbb{T})$.
Then the Hausdorff-Young theorem (\cite[13.5.1]{E}) yields 
$\Vert\widehat{f} \, \Vert_q\le\Vert f\Vert_p$.
By H\"{o}lder's inequality, we have
\[
\sum_{k=-\infty}^{+\infty}\bigl\vert a(k)\widehat{f}(k)^2\bigr\vert\le
\left\Vert a \right\Vert_r\bigl\Vert\widehat{f}^{\ 2}\bigr\Vert_s =
\left\Vert a\right\Vert_r\bigl\Vert\widehat{f} \, \bigr\Vert_q^2\le
\left\Vert a\right\Vert_r \left\Vert f \right\Vert_p^2.
\]
This allows us to define an orthogonally additive continuous $2$-homogeneous polynomial $P\colon L^p(\mathbb{T})\to\mathbb{C}$ by
\[
P(f)=
\sum_{k=-\infty}^{+\infty}a(k)\widehat{f\ast f}(k)=
\sum_{k=-\infty}^{+\infty}a(k)\widehat{f}(k)^2 \quad (f\in L^p(\mathbb{T})).
\]
Suppose that $P$ can be expressed as $P(f)=\Phi(f\ast f)$ $(f\in L^p(\mathbb{T}))$ 
for some continuous linear functional $\Phi\colon L^p(\mathbb{T})\to\mathbb{C}$.
By \eqref{ca00} and \eqref{ca0}, we have
\[
\Phi(\chi_k)=\Phi(\chi_k\ast \chi_k)=P(\chi_k)=
\sum_{j=-\infty}^{+\infty}a(j)\widehat{\chi_k}(j)=
\sum_{j=-\infty}^{+\infty}a(j){\delta_{jk}}=
a(k)
\]
for each $k\in\mathbb{Z}$.
If $f\in L^p(\mathbb{T})$, then Riesz's theorem~\cite[12.10.1]{E} shows that the sequence
$(\sum_{k=-m}^m\widehat{f}(k)\chi_k)$ converges to $f$ in $L^p(\mathbb{T})$. Since $\Phi$ is continuous,
it follows that the sequence
\[
\left(\Phi\left(\sum_{k=-m}^m\widehat{f}(k)\chi_k\right)\right)=
\left(\sum_{k=-m}^m\widehat{f}(k)\Phi(\chi_k)\right)=
\left(\sum_{k=-m}^m\widehat{f}(k)a(k)\right)
\]
converges to $\Phi(f)$. 
In particular, the sequence
$\bigl(\sum_{k=-m}^m\widehat{h}(k)a(k)\bigr)$ is convergent, 
which contradicts the choice of $h$.
\end{example}

\begin{example}\label{ca12}
Assume that $2\le p<\infty$.
If $f\in L^p(\mathbb{T})$, then $f\in L^2(\mathbb{T})$ and therefore 
$\Vert\widehat{f} \, \Vert_2=\Vert f\Vert_2\le\Vert f\Vert_p$.
This allows us to define an orthogonally additive continuous $2$-homogeneous polynomial $P\colon L^p(\mathbb{T})\to\mathbb{C}$ by
\begin{equation*}
P(f)=
\sum_{k=-\infty}^{+\infty}\widehat{f\ast f}(k) =
\sum_{k=-\infty}^{+\infty}\widehat{f}(k)^2 \quad (f\in L^p(\mathbb{T})).
\end{equation*}
Suppose that $P$ can be represented in the form 
$P(f)=\Phi(f\ast f)$ $(f\in L^p(\mathbb{T}))$ for some continuous linear functional 
$\Phi\colon L^p(\mathbb{T})\to\mathbb{C}$. Let $h\in L^q(\mathbb{T})$ be such that
\begin{equation}\label{ca3}
\Phi(f)=\int_\mathbb{T}f(z)h(z) \, dz \quad (f\in L^p(\mathbb{T})).
\end{equation}
By \eqref{ca00} and \eqref{ca0}, for each $k\in\mathbb{Z}$, we have
\[
\Phi(\chi_k)=\Phi(\chi_k\ast \chi_k)=P(\chi_k)=
\sum_{j=-\infty}^{+\infty}\widehat{\chi_k}(j)= 
\sum_{j=-\infty}^{+\infty}{\delta_{jk}}= 1,
\]
and \eqref{ca3} then yields
\begin{equation*}
\widehat{h}(k)=\int_\mathbb{T}z^{-k}h(z) \, dz=\Phi(\chi_{-k})=1,
\end{equation*}
contrary to  Riemann-Lebesgue lemma.
\end{example}

\begin{example}
If $f\in L^\infty(\mathbb{T})$, 
then $f\in L^2(\mathbb{T})$ and therefore
$\Vert\widehat{f} \, \Vert_2=\Vert f\Vert_2\le\Vert f\Vert_\infty$, which implies that 
\[
\sum_{k=0}^{+\infty}\bigl\vert\widehat{f}(-k)\bigr\vert^2\le
\sum_{k=-\infty}^{+\infty}\bigl\vert\widehat{f}(k)\bigr\vert^2\le\Vert f\Vert_\infty^2.
\]
Hence we can define an orthogonally additive continuous $2$-homogeneous polynomial 
$P\colon L^\infty(\mathbb{T})\to\mathbb{C}$ by
\[
P(f)=
\sum_{k=0}^{+\infty}\widehat{f\ast f}(-k)=
\sum_{k=0}^{+\infty}\widehat{f}(-k)^2 \quad (f\in L^\infty(\mathbb{T})).
\]
Suppose that $P$ can be represented in the form 
$P(f)=\Phi(f\ast f)$ $(f\in L^\infty(\mathbb{T}))$ for some continuous linear functional 
$\Phi\colon L^\infty(\mathbb{T})\to\mathbb{C}$.
The restriction of $\Phi$ to $C(\mathbb{T})$ gives a continuous linear functional on $C(\mathbb{T})$
and therefore there exists a measure $\mu\in M(\mathbb{T})$ such that
\begin{equation}\label{ca2}
\Phi(f)=\int_\mathbb{T}f(z) \, d\mu(z) \quad (f\in C(\mathbb{T})).
\end{equation}
By \eqref{ca00} and \eqref{ca0}, for each $k\in\mathbb{Z}$, we have
\[
\Phi(\chi_k)=\Phi(\chi_k\ast\chi_k)=P(\chi_k)=
\sum_{j=0}^{+\infty}\widehat{\chi_k}(-j)= 
\sum_{j=0}^{+\infty}{\delta_{-jk}}= 
\begin{cases}
1 & \text{if $k\le 0$,}\\
0 & \text{if $k>0$,} 
\end{cases}
\]
and \eqref{ca2} then yields
\begin{equation*}
\widehat{\mu}(k)=
\int_\mathbb{T}z^{-k} \, d\mu (z) \, dz=
\Phi(\chi_{-k})=
\begin{cases}
1 & \text{if $k\ge 0$,}\\
0 & \text{if $k<0$.}
\end{cases}
\end{equation*}
This contradicts the fact that the series $\sum_{k\ge 0}\chi_k$ is not
a Fourier-Stieltjes series (see \cite[Example~12.7.8]{E}).
It should be pointed out that we have actually shown that neither $P$ nor the restriction of $P$ to $C(\mathbb{T})$
can be represented in the standard form.
\end{example}

\section{The largest appropriate norm topology on $\mathcal{P}_n(A)$}

Since Question Q2 has been settled in the negative for the algebras $L^p(G)$, with $1<p\le\infty$, and $C(G)$,
it is therefore reasonable to attempt to explore Question Q3 for them.
For this purpose, 
in this section, for each Banach algebra $A$, we make an appropriate choice of norm on $\mathcal{P}_n(A)$.

\begin{theorem}\label{g}
Let $A$ be a Banach algebra.
Then 
\begin{equation*}
\mathcal{P}_n(A)=
\biggl\{
\sum_{j=1}^m a_j^n :  a_1,\dotsc,a_m\in A, \ m\in\mathbb{N}
\biggr\}
\end{equation*}
and the formula
\begin{equation*}
\Vert a\Vert_{\mathcal{P}_n}=
\inf\biggl\{
\sum_{j=1}^m \left\Vert a_j \right\Vert^n : a=\sum_{j=1}^m a_j^n
\biggr\},
\end{equation*}
for each $a\in\mathcal{P}_n(A)$, defines a norm on $\mathcal{P}_n(A)$
such that
\begin{equation*}
\Vert a^n\Vert_{\mathcal{P}_n}\le \left\Vert a\right\Vert^n \quad (a\in A)
\end{equation*}
and
\begin{equation*}
\left\Vert a \right\Vert \le \left\Vert a \right\Vert_{\mathcal{P}_n} \quad (a\in\mathcal{P}_n(A)).
\end{equation*}
Further, the following statements hold.
\begin{enumerate}
\item
Suppose that $\Phi\colon\mathcal{P}_n(A)\to X$ is a $\norm{\cdot}_{\mathcal{P}_n}$-continuous linear map for some Banach space $X$.
Then the map $P\colon A\to X$ defined by $P(a)=\Phi \left(a^n \right)$ $(a\in A)$ is an  orthogonally additive continuous
$n$-homogeneous polynomial with $\Vert P\Vert=\Vert\Phi\Vert$.
\item
Suppose that $\tnorma{\cdot}$ is a norm on $\mathcal{P}_n(A)$
for which the answer to Question Q3 is positive. Then there exist
$M_1,M_2\in\mathbb{R}^+$ such that
\[
M_1 \left\Vert a \right\Vert \le \vert\vert\vert a\vert\vert\vert 
\le M_2 \left\Vert a \right\Vert_{\mathcal{P}_n}
\quad (a\in\mathcal{P}_n(A)).
\]
\end{enumerate}
\end{theorem}

\begin{proof}
Let $a_1,\dotsc,a_m\in A$ and $\alpha_1,\dotsc,\alpha_m\in\mathbb{C}$.
Take $\beta_1,\dotsc,\beta_m\in\mathbb{C}$ such that $\beta_j^n=\alpha_j$ $(j\in\{1,\dotsc,m\})$.
Then 
\[
\sum_{j=1}^m\alpha_j a_j^n=\sum_{j=1}^m \left( \beta_j a_j \right)^n,
\]
which establishes the first equality of the result.

Take $a\in\mathcal{P}_n(A)$, and let $a_1,\dotsc,a_m\in A$ be such that $a=\sum_{j=1}^m a_j^n$. 
Then
\[
\left\Vert a \right\Vert 
\le \sum_{j=1}^m \left\Vert a_j^n \right\Vert 
\le \sum_{j=1}^m \left\Vert a_j \right\Vert^n,
\]
which proves that $\left\Vert a \right\Vert \le \left\Vert a \right\Vert_{\mathcal{P}_n}$.
In particular, if $a\in A$ is such that $\left\Vert a \right\Vert_{\mathcal{P}_n}=0$, then we have $a=0$.

Set $a\in\mathcal{P}_n(A)$ and $\alpha\in\mathbb{C}$.
We proceed to show that $\left\Vert \alpha a \right\Vert_{\mathcal{P}_n}=\left\vert \alpha \right\vert \left\Vert a \right\Vert_{\mathcal{P}_n}$.
Of course, we can assume that $\alpha\ne 0$.
Choose $\beta\in\mathbb{C}$ such that $\beta^n=\alpha$.
If $a_1,\dotsc,a_m\in A$ are such that $a=\sum_{j=1}^m a_j^n$
then
$\alpha a=\sum_{j=1}^m \left( \beta a_j\right)^n$ and therefore
\[
\left\Vert\alpha a\right\Vert_{\mathcal{P}_n}\le
\sum_{j=1}^m \left\Vert\beta a_j \right\Vert^n=
\sum_{j=1}^m \left\vert \alpha \right\vert \left\Vert a_j\right\Vert^n,
\]
which implies that $\left\Vert \alpha a \right\Vert_{\mathcal{P}_n}\le \left\vert \alpha \right\vert \left\Vert a \right\Vert_{\mathcal{P}_n}$.
On the other hand, 
\[
\left\Vert a \right\Vert_{\mathcal{P}_n}=
\left\Vert \alpha^{-1}(\alpha a) \right\Vert_{\mathcal{P}_n}\le
\left\vert\alpha\right\vert^{-1} \left\Vert\alpha a \right\Vert_{\mathcal{P}_n},
\]
which gives the converse inequality.

Let $a$, $b\in\mathcal{P}_n(A)$. 
Our goal is to prove that 
$\left\Vert a+b \right\Vert_{\mathcal{P}_n}\le \left\Vert a \right\Vert_{\mathcal{P}_n}+ \left\Vert b \right\Vert_{\mathcal{P}_n}$.
To this end, set $\varepsilon\in\mathbb{R}^+$, and choose 
$a_1,\dotsc,a_l,b_1,\dotsc,b_m\in A$ such that
\[
a=\sum_{j=1}^l a_j^n, \quad b=\sum_{j=1}^m b_j^n,
\]
and
\[
\sum_{j=1}^l \left\Vert a_j\right\Vert^n<\left\Vert a\right\Vert_{\mathcal{P}_n}+\varepsilon/2, \quad
\sum_{j=1}^m \left\Vert b_j\right\Vert^n<\left\Vert b\right\Vert_{\mathcal{P}_n}+\varepsilon/2.
\]
Then we have
\[
a+b=\sum_{j=1}^l a_j^n+\sum_{j=1}^m b_j^n
\]
and therefore
\[
\left\Vert a+b \right\Vert_{\mathcal{P}_n}\le
\sum_{j=1}^l \left\Vert a_j \right\Vert^n + \sum_{j=1}^m \left\Vert b_j\right \Vert^n\le
\left\Vert a \right\Vert_{\mathcal{P}_n} + \left\Vert b \right\Vert_{\mathcal{P}_n}+\varepsilon,
\]
which yields $\left\Vert a+b \right\Vert_{\mathcal{P}_n}\le \left\Vert a \right\Vert_{\mathcal{P}_n} + \left\Vert b \right\Vert_{\mathcal{P}_n}$.
Then $\norm{\cdot}_{\mathcal{P}_n}$ is a norm on $\mathcal{P}_n(A)$.
The space $\mathcal{P}_n(A)$ is equipped with this norm for the remainder of this proof. 

It is clear that $\Vert a^n\Vert_{\mathcal{P}_n}\le\Vert a\Vert^n$ for each $a\in A$.
This property allows us to establish the statement $(1)$.
Suppose that $X$ is a Banach space and that $\Phi\colon\mathcal{P}_n(A)\to X$ is a continuous linear map.
Define $P\colon A\to X$ by $P(a)=\Phi \left(a^n \right)$ $(a\in A)$.
Take $a\in A$.
Then we have
\[
\norm{P(a)} =
\norm{\Phi\left( a^n \right) } \le
\left\Vert \Phi \right\Vert \left\Vert a^n \right\Vert_{\mathcal{P}_n}\le
\left\Vert \Phi \right\Vert \left\Vert a\right\Vert^n.
\]
which shows that the polynomial $P$ is continuous and that $\Vert P\Vert\le\Vert\Phi\Vert$.
On the other hand, let $a_1,\dotsc,a_m\in A$ be such that $a=\sum_{j=1}^m a_j^n$. Then we have
\[
\Phi(a)=\sum_{j=1}^m\Phi(a_j^n)=\sum_{j=1}^m P(a_j),
\]
whence
\[
\Vert\Phi(a)\Vert\le
\sum_{j=1}^m \left\Vert P(a_j)\right\Vert \le
\sum_{j=1}^m \left\Vert P \right\Vert \left\Vert a_j\right\Vert^n=
\Vert P\Vert\sum_{j=1}^m \left\Vert a_j \right\Vert^n.
\]
This shows that $\left\Vert \Phi(a) \right\Vert \le \left\Vert P \right\Vert \left\Vert a \right\Vert_{\mathcal{P}_n}$,
hence that $\left\Vert \Phi \right\Vert \le \left\Vert P \right\Vert$, and finally that $\Vert\Phi\Vert=\Vert P\Vert$. 

We now prove $(2)$.
Suppose that $\tnorma{\cdot}$ is a norm on $\mathcal{P}_n(A)$ which satisfies
the conditions:
\begin{itemize}
\item
for each Banach space $X$ and each $\tnorma{\cdot}$-continuous linear map 
$\Phi\colon\mathcal{P}_n(A)\to X$, the prototypical polinomial $P\colon A\to X$ defined by 
$P(a)=\Phi\left( a^n \right) $ $(a\in A)$ is continuous, and
\item
every  orthogonally additive continuous $n$-homogeneous polynomial $P$ from $A$
into each Banach space $X$ can be expressed as $P(a)=\Phi\left( a^n \right) $ $(a\in A)$  for some
$\tnorma{\cdot}$-continuous linear map 
$\Phi\colon\mathcal{P}_n(A)\to X$.
\end{itemize}
The canonical polynomial $P_n\colon A\to A$ is an  orthogonally additive
continuous $n$-homogeneous polynomial.
Hence there exists a $\tnorma{\cdot}$-continuous linear map 
$\Phi\colon\mathcal{P}_n(A)\to A$ such that $P_n(a)=\Phi\left( a^n \right) $ $(a\in A)$.
This implies that $\Phi$ is the inclusion map from $\mathcal{P}_n(A)$ into $A$,
and the continuity of this map yields $M_1\in\mathbb{R}^+$ such that 
$M_1\Vert a\Vert\le\vert\vert\vert a\vert\vert\vert$ for each $a\in\mathcal{P}_n(A)$.
We now take $X$ to be the completion of the normed space $(\mathcal{P}_n(A),\tnorma{\cdot})$,
and let $\Phi\colon \mathcal{P}_n(A)\to X$ be the inclusion map.
Then, by hypothesis, the polynomial  $P\colon A\to A$ defined by $P(a)=\Phi\left( a^n \right) $ $(a\in A)$ is continuous, and,
in consequence, there exists $M_2\in\mathbb{R}^+$ such that
$\vert\vert\vert a^n\vert\vert\vert\le M_2\Vert a\Vert^n$ for each $a\in A$.
Take $a\in\mathcal{P}_n(A)$, and 
let $a_1,\dotsc,a_m\in A$ be such that $a=\sum_{j=1}^m a_j^n$. Then we have
\[
\vert\vert\vert a\vert\vert\vert\le
\sum_{j=1}^m\vert\vert\vert a_j^n \vert\vert\vert\le
\sum_{j=1}^m M_2 \left\Vert a_j \right\Vert^n,
\]
which yields $\vert\vert\vert a\vert\vert\vert\le\Vert a\Vert_{\mathcal{P}_n}$,
and completes the proof.
\end{proof}

From now on $\mathcal{P}_n(A)$ will be equipped with the norm $\lVert\cdot\rVert_{\mathcal{P}_n}$
defined in Theorem~\ref{g}.

\begin{proposition}\label{p1827}
Let $A$ be a Banach algebra.
Then
\[
\mathcal{P}_n(A)=
\biggl\{
\sum_{j=1}^m S_n(a_{1,j},\dotsc,a_{n,j}) :  a_{1,j},\dotsc,a_{n,j}\in A, \ m\in\mathbb{N}
\biggr\}
\]
and the formula
\begin{equation*}
\Vert a\Vert_{\mathcal{S}_n}=
\inf\biggl\{
\sum_{j=1}^m \left\Vert a_{1,j} \right\Vert \cdots \left\Vert a_{n,j} \right\Vert : a=
\sum_{j=1}^m S_n(a_{1,j},\dotsc,a_{n,j})
\biggr\}, 
\end{equation*}
for each $a\in\mathcal{P}_n(A)$,
defines a norm on $\mathcal{P}_n(A)$ such that
\begin{equation*}
\left\Vert a \right\Vert \le
\left\Vert a \right\Vert_{\mathcal{S}_n}\le
\left\Vert a \right\Vert_{\mathcal{P}_n}\le
\frac{n^n}{n!} \left\Vert a \right\Vert_{\mathcal{S}_n}
\quad \left(a\in\mathcal{P}_n(A)\right).
\end{equation*}
\end{proposition}

\begin{proof}
Let $\mathcal{S}_n(A)$ denote the set of the right-hand side of the first identity in the result.
Take $a\in\mathcal{P}_n(A)$.
Then there exists $a_1,\dotsc,a_m\in A$ such that 
\[
a=\sum_{j=1}^m a_j^n=\sum_{j=1}^m S_n(a_j,\dotsc,a_j)\in\mathcal{S}_n(A).
\]
This also implies that
$\left\Vert a \right\Vert_{\mathcal{S}_n}\le\sum_{j=1}^m \left\Vert a_j \right\Vert \dotsb \left\Vert a_j \right\Vert$,
and hence that 
$\left\Vert a \right\Vert_{\mathcal{S}_n}\le \left\Vert a \right\Vert_{\mathcal{P}_n}$.
We now take $a\in\mathcal{S}_n(A)$. 
Then there exist $a_{1,j},\dotsc,a_{n,j}\in A$ $(j\in\{1,\dotsc,m\})$ such that
$a=\sum_{j=1}^m S_n(a_{1,j},\dotsc,a_{n,j})$.
The polarization formula gives
\begin{equation*}
\begin{split}
a & =
\sum_{j=1}^m
\frac{1}{n!\,2^n}\sum_{\epsilon_1,\ldots,\epsilon_n=\pm 1}
\epsilon_{1}\cdots\epsilon_{n} \left(\epsilon_1a_{1,j}+\dots+\epsilon_{n}a_{n,j} \right)^n\\
& =
\sum_{j=1}^m
\sum_{\epsilon_1,\ldots,\epsilon_n=\pm 1}
\Bigl(
\bigl(\tfrac{1}{n!\,2^n}\bigr)^{1/n}\epsilon_{1}^{1/n}\cdots\epsilon_{n}^{1/n}\left(\epsilon_1a_{1,j}+\dotsb+\epsilon_{n}a_{n,j} \right)\Bigr)^n
\in
\mathcal{P}_n(A).
\end{split}
\end{equation*}

The proof of the fact that $\left\Vert \cdot \right\Vert_{\mathcal{S}_n}$ is a norm on $\mathcal{P}_n(A)$ 
is similar to Theorem~\ref{g}, and it is omitted.

Our next objective is to prove the inequalities relating the three norms 
$\norm{\cdot}$,
$\norm{\cdot}_{\mathcal{S}_n}$, and
$\norm{\cdot}_{\mathcal{P}_n}$.
Take $a\in\mathcal{P}_n(A)$.
We have already shown that $\left\Vert a \right\Vert_{\mathcal{S}_n} \le 
\left\Vert a \right\Vert_{\mathcal{P}_n}$.
Let $a_{1,j},\dotsc, a_{n,j}\in A$ $(j\in\{1,\dotsc,m\})$
be such that $a=\sum_{j=1}^m S_n(a_{1,j},\dotsc, a_{n,j})$.
Then
\[
\norm{a} \le
\sum_{j=1}^m \left\Vert S_n \left(a_{1,j},\dotsc, a_{n,j} \right) \right\Vert
\le \sum_{j=1}^m \left\Vert a_{1,j} \right\Vert \dotsb \left\Vert a_{n,j} \right\Vert,
\]
which shows that $\norm{a} \le \left\Vert a \right\Vert_{\mathcal{S}_n}$.
For $l\in\left\{1,\dotsc, n \right\}$ and $j\in \left\{ 1,\dotsc, m \right\}$,
we take $b_{l,j}=a_{l,j}/\left\Vert a_{l,j} \right\Vert$ if $a_{l,j}\ne 0$ and $b_{l,j}=0$ otherwise.
Then
\[
a=\sum_{j=1}^m S_n \left(a_{1,j},\dotsc,a_{n,j} \right) =
\sum_{j=1}^m \left\Vert a_{1,j} \right\Vert \dotsb \left\Vert a_{n,j} \right\Vert S_n \left(b_{1,j},\dotsc, b_{n,j} \right),
\]
and the polarization formula gives
\begin{multline*}
a=
\sum_{j=1}^m
\left\Vert a_{1,j} \right\Vert \dotsb \left\Vert a_{n,j} \right\Vert
\frac{1}{n!\,2^n}\sum_{\epsilon_1,\ldots,\epsilon_n=\pm 1}
\epsilon_{1}\cdots\epsilon_{n} \left( \epsilon_1b_{1,j}+ \dotsb +\epsilon_{n}b_{n,j} \right)^n \\
= \sum_{j=1}^m\sum_{\epsilon_1,\ldots,\epsilon_n=\pm 1}
\Bigl(\bigl( \left\Vert a_{1,j} \right\Vert \dotsb \left\Vert a_{n,j} \right\Vert
\tfrac{1}{n!\,2^n}\bigr)^{1/n}
\epsilon_{1}^{1/n} \dotsb \epsilon_{n}^{1/n}
\bigl(\epsilon_1b_{1,j}+\dotsb +\epsilon_{n}b_{n,j}\bigr)\Bigr)^{\!n}.
\end{multline*}
We thus get
\begin{align*}
\left\Vert a \right\Vert_{\mathcal{P}_n}  & \le
\sum_{j=1}^m\sum_{\epsilon_1,\ldots,\epsilon_n=\pm 1}
\Bigl\Vert\bigl( \left\Vert a_{1,j} \right\Vert \dotsb \left\Vert a_{n,j} \right\Vert
\tfrac{1}{n!\,2^n}\bigr)^{1/n}
\epsilon_{1}^{1/n}\dotsb\epsilon_{n}^{1/n}
\bigl(\epsilon_1b_{1,j}+ \dotsb +\epsilon_{n}b_{n,j}\bigr)\Bigr\Vert^n \\
& = 
\sum_{j=1}^m\sum_{\epsilon_1,\ldots,\epsilon_n=\pm 1}
\left\Vert a_{1,j} \right\Vert \dotsb \left\Vert a_{n,j} \right\Vert
\frac{1}{n!\,2^n}
\bigl\Vert\epsilon_1b_{1,j}+\dotsb +\epsilon_{n}b_{n,j}\bigr\Vert^n \\
& \le
\sum_{j=1}^m\sum_{\epsilon_1,\ldots,\epsilon_n=\pm 1}
\left\Vert a_{1,j} \right\Vert \dotsb \left\Vert a_{n,j} \right\Vert
\frac{1}{n!\,2^n} n^n= 
\sum_{j=1}^m \left\Vert a_{1,j} \right\Vert \dotsb \left\Vert a_{n,j} \right\Vert
\frac{n^n}{n!},
\end{align*}
which shows that $\norm{a}_{\mathcal{P}_n}\le\frac{n^n}{n!}\norm{a}_{\mathcal{S}_n}$.
\end{proof}

\begin{proposition}\label{p19}
Let $A$ be a Banach algebra with a central bounded approximate identity 
of bound $M$. Then $\mathcal{P}_n(A)=A$ and
$\norm{\cdot} \le \norm{\cdot}_{\mathcal{S}_n} \le M^{n-1} \norm{\cdot}$.
\end{proposition}

\begin{proof}
Take $a\in A$, and let $\varepsilon\in\mathbb{R}^+$.
Our objective is to show that $a\in \mathcal{P}_n(A)$ and that 
$\left\Vert a \right\Vert_{\mathcal{S}_n}<M^{n-1}\bigl(\left\Vert a 
\right\Vert +\varepsilon\bigr)$. Let $\mathcal{Z}(A)$ be the centre of $A$,
and let $(e_\lambda)_{\lambda\in\Lambda}$ be a central bounded approximate identity for $A$ of bound $M$.
Then $\mathcal{Z}(A)$ is a Banach algebra and  $(e_\lambda)_{\lambda\in\Lambda}$ is a bounded approximate identity for $\mathcal{Z}(A)$. Of course, $A$ is 
a Banach $\mathcal{Z}(A)$-bimodule and $\mathcal{Z}(A)A$ is dense in $A$.
We now take $a_0=a$, and then we successively apply the factorization 
theorem~\cite[Theorem~2.9.24]{D} to choose
$a_1,\dotsc,a_{n-1}\in A$ and $z_1,\dotsc,z_{n-1}\in\mathcal{Z}(A)$ such that
\begin{gather}
a_{k-1}=z_k a_k, \\
\left\Vert a_{k-1}-a_{k} \right\Vert < \frac{\varepsilon}{n},
\end{gather}
and 
\begin{equation}
\Vert z_k\Vert\le M
\end{equation}
for each $k\in\{1,\dotsc,n-1\}$. 
Then
\[
a=S_n \left(z_1,\dotsc,z_{n-1},a_{n-1} \right)
\]
and therefore $a\in\mathcal{P}_n(A)$ with
\begin{equation*}
\begin{split}
\left\Vert a \right\Vert_{\mathcal{S}_n} & \leq
\left\Vert z_1 \right\Vert \dotsb \left\Vert z_{n-1} \right\Vert \left\Vert a_{n-1} \right\Vert \\ 
& \leq M^{n-1} \left\Vert a_{n-1} \right\Vert \\
& \leq M^{n-1}\bigl(
\left\Vert a_{n-1}-a_{n-2} \right\Vert + \dotsb + \left\Vert a_1-a \right\Vert + \left\Vert a \right\Vert\bigr)\\
& < M^{n-1}\bigl( \left\Vert a \right\Vert+\varepsilon\bigr),
\end{split}
\end{equation*}
as claimed.
Further, on account of Theorem~\ref{g}, we have
\[
\left\Vert a \right\Vert \le \left\Vert a \right\Vert_{\mathcal{S}_n} \le M^{n-1}\bigl( \left\Vert a \right\Vert+\varepsilon\bigr)
\] 
for each $\varepsilon\in\mathbb{R}^+$,
which gives $\left\Vert a \right\Vert \le \left\Vert a \right\Vert_{\mathcal{S}_n}\le M^{n-1} \left\Vert a \right\Vert$.
\end{proof}

\begin{corollary}\label{p1}
Let $G$ be a compact group.
Then $\mathcal{P}_n\bigl(L^1(G)\bigr)=L^1(G)$ and 
$\left\Vert \cdot \right\Vert_{\mathcal{S}_n}= \left\Vert \cdot \right\Vert_1$.
\end{corollary}

\begin{proof}
The net $(h_\lambda)_{\lambda\in\Lambda}$ introduced in the beginning of Section~\ref{sect}
is a central bounded approximate identity for $L^1(G)$ of bound $1$. Then
Proposition~\ref{p19} applies to show the result.
\end{proof}

\section{Some other examples of convolution algebras}

Sections 2 and 5 combine to answer Question Q3 in the positive 
for a variety of convolutions algebras associated with $G$.

\begin{lemma}\label{1319}
Let $G$ be a compact group, 
and let $A$ be a subalgebra of $L^1(G)$ which is equipped with a norm $\left\Vert \cdot \right\Vert_A$ of its own and
satisfies the following conditions:
\begin{enumerate}
\item[(a)]
$A$ is a Banach algebra
with respect to  $\left\Vert \cdot \right\Vert_A$; 
\item[(b)]
$\mathcal{T}(G)$ is a dense subspace of $A$
with respect to  $\left\Vert \cdot \right\Vert_A$;
\end{enumerate}
Then $\mathcal{T}(G)$ is a dense subset of $\mathcal{P}_n(A)$.
\end{lemma}

\begin{proof}
For $\left[\pi \right]\in\widehat{G}$ and $f\in\mathcal{T}_\pi(G)$, the polarization formula and Lemma~\ref{l1}(1) yield
\begin{equation*}
\begin{split}
f & = S_n \left(f,d_\pi\chi_\pi,\dotsc,d_\pi\chi_\pi \right)\\
& =
\frac{1}{n!\,2^n}\sum_{\epsilon_1,\ldots,\epsilon_n=\pm 1}
\epsilon_{1}\cdots\epsilon_{n} \left(\epsilon_1f+\epsilon_2 d_\pi\chi_\pi+\dotsb +\epsilon_{n}d_\pi\chi_\pi \right)^{*n}.
\end{split}
\end{equation*}
This shows that
$\mathcal{T}_\pi(G)$ is equal to the linear span of the set $\bigl\{f^{*n} :  f\in\mathcal{T}_\pi(G)\bigr\}$,
which implies that $\mathcal{T}(G)\subset\mathcal{P}_n(A)$.

On the other hand, if $f,g\in A$, then
\begin{equation*}
\begin{split}
f^{*n}
 & =
P_n\bigl(g+(f-g)\bigr)\\
& =
\sum_{j=0}^{n}\binom{n}{j} S_n (\underbrace{g,\dotsc,g}_{n-j},\underbrace{f-g,\dotsc,f-g}_{j})\\
& =
g^{*n}+\sum_{j=1}^{n}\binom{n}{j} S_n(\underbrace{g,\dotsc,g}_{n-j},\underbrace{f-g,\dotsc,f-g}_{j})
\end{split}
\end{equation*}
and therefore
\begin{equation*}
\left\Vert f^{*n}-g^{*n} \right\Vert_{\mathcal{S}_n}\le
\sum_{j=1}^{n}\binom{n}{j} \left\Vert g \right\Vert_{A}^{n-j} \left\Vert f-g \right\Vert_{A}^{j}.
\end{equation*}
From Proposition~\ref{p1827}, we deduce that
\begin{equation}\label{1452}
\left\Vert f^{*n}-g^{*n} \right\Vert_{\mathcal{P}_n}\le
\frac{n^n}{n!}
\sum_{j=1}^{n}\binom{n}{j} \left\Vert g \right\Vert_{A}^{n-j} \left\Vert f-g \right\Vert_{A}^{j}.
\end{equation}
Let $f\in A$. Then there exists a sequence $(f_k)$ in $\mathcal{T}(G)$ such that $\left\Vert f-f_k \right\Vert_A\to 0$,
and~\eqref{1452} then gives $\left\Vert f^{*n}-f_k^{*n} \right\Vert_{\mathcal{P}_n}\to 0$.
This implies that $\mathcal{T}(G)$ is dense in $\mathcal{P}_n(A)$.
\end{proof}

\begin{theorem}\label{tf}
Let $G$ be a compact group,
and let $A$ be a subalgebra of $L^1(G)$ which is equipped with a norm $\left\Vert \cdot \right\Vert_A$ of its own and
satisfies the following conditions:
\begin{enumerate}
\item[(a)]
$A$ is a Banach algebra
with respect to $\left\Vert \cdot \right\Vert_A$; 
\item[(b)]
$\mathcal{T}(G)$ is a dense subspace of $A$
with respect to $\left\Vert \cdot \right\Vert_A$;
\item[(c)]
$A$ is a Banach left $L^1(G)$-module with respect to $\left\Vert \cdot \right\Vert_A$ and the convolution multiplication.
\end{enumerate}
Let $X$ be a Banach space, and
let $P\colon A\to X$ be a continuous $n$-homogeneous polynomial.
Then the following conditions are equivalent:
\begin{enumerate}
\item
the polynomial $P$ is orthogonally additive;
\item
the polynomial $P$ is orthogonally additive on $\mathcal{T}(G)$, i.e.,
$P(f+g)=P(f)+P(g)$ whenever $f,g\in\mathcal{T}(G)$ are such that $f\ast g=g\ast f=0$;
\item
there exists a unique continuous linear map $\Phi\colon\mathcal{P}_n(A)\to X$ such that
$P(f)=\Phi\left( f^{*n} \right)$ for each $f\in A$.
\end{enumerate}
\end{theorem}

\begin{proof}
It is clear that $(1)\Rightarrow(2)$ and that $(3)\Rightarrow(1)$.
We will prove that $(2)\Rightarrow(3)$.

Let $\varphi\colon A^n\to X$ be the symmetric $n$-linear map associated with $P$,
and let $\Phi_0\colon\mathcal{T}(G)\to X$ be the linear map defined by  
\begin{equation*}
\Phi_0(f)=
\sum_{[\pi]\in\widehat{G}}\varphi \left( d_\pi f\ast\chi_\pi,d_\pi\chi_\pi\dotsc,d_\pi \chi_\pi \right)
\end{equation*}
for each $f\in\mathcal{T}(G)$.
Since $P$ is orthogonally additive on $\mathcal{T}(G)$, Theorem~\ref{t1} yields
\begin{equation}\label{e1056}
P(f)=\Phi_0 \left(f^{*n} \right) \quad \left( f\in\mathcal{T}(G) \right).
\end{equation}
We claim that $\Phi_0$ is continuous.
Let $(h_\lambda)_{\lambda\in\Lambda}$ be as introduced in the beginning of Section~\ref{sect}.
Set $f\in\mathcal{T}(G)$, and 
assume that $f=\sum_{j=1}^m f_j^{*n}$ with $f_1,\dotsc ,f_m\in A$.
For $\lambda\in\Lambda$, $h_\lambda$ belongs to the centre of $L^1(G)$, and so
\[
h_\lambda^{*n}\ast f=\sum_{j=1}^m \bigl( h_\lambda\ast f_j \bigr)^{*n}.
\]
Since $f_j\ast h_\lambda\in\mathcal{T}(G)$ $(j\in\{1,\dotsc,m\}, \ \lambda\in\Lambda)$, \eqref{e1056} yields
\[
\Phi_0 \left( h_\lambda^{*n}\ast f \right)=
\sum_{j=1}^m\Phi_0\bigl( \left(h_\lambda\ast f_j \right)^{*n}\bigr)=
\sum_{j=1}^m P \left( h_\lambda\ast f_j \right),
\]
whence
\begin{equation*}
\begin{split}
\left\Vert \Phi_0 \left( h_\lambda^{*n}\ast f \right) \right\Vert & =
\sum_{j=1}^m \left\Vert P(h_\lambda\ast f_j)\right\Vert  \leq 
\sum_{j=1}^m \left\Vert P\right\Vert \left\Vert h_\lambda\ast f_j \right\Vert_A^n \\ 
&\leq
\sum_{j=1}^m \Vert P\Vert \left\Vert h_\lambda\right\Vert_1^n \left\Vert f_j\right\Vert_A^n  \leq
\left\Vert P \right\Vert\sum_{j=1}^m \left\Vert f_j\right\Vert_A^n.    
\end{split}
\end{equation*}
We thus get
\begin{equation}\label{1126}
\left\Vert \Phi_0 \left( h_\lambda^{*n}\ast f \right) \right\Vert \le 
\left\Vert P \right\Vert \left\Vert f \right\Vert_{\mathcal{P}_n}.
\end{equation}
We now see that, for each $\lambda\in\Lambda$,
\begin{equation*}
\begin{split}
\left\Vert h_\lambda^{*n}\ast f-f\right\Vert_A & \le
\left\Vert h_\lambda^{*n}\ast f-h_\lambda^{*n-1}\ast f\right\Vert_A\\
& \quad {}+ \dotsb +
\left\Vert h_\lambda^{*2}\ast f-h_\lambda\ast f \right\Vert_A+ \left\Vert h_\lambda\ast f-f\right\Vert_A
\\ 
& \le
\bigl(\left\Vert h_\lambda\right\Vert_1^{n-1}+\dots+ \left\Vert h_\lambda\right\Vert_1+1\bigr) \left\Vert h_\lambda\ast f-f\right\Vert_A
\\ 
& \le
n\left\Vert h_\lambda\ast f-f\right\Vert_A.
\end{split}
\end{equation*}

On account of \cite[Remarks~32.33(a) and 38.6(b)]{HR}, we have
\[
\lim_{\lambda\in\Lambda} \left\Vert h_\lambda\ast f-f \right\Vert_A=0,
\]
and so
\[
\lim_{\lambda\in\Lambda} \left\Vert h_\lambda^{*n}\ast f-f \right\Vert_A=0.
\]
Since $f\in\mathcal{T}(G)$, it follows that there exist $[\pi_1],\dotsc,[\pi_l]\in\widehat{G}$
such that $f\in\mathcal{M}:=\mathcal{T}_{\pi_1}(G)+\dotsb+\mathcal{T}_{\pi_l}(G)$.
The finite-dimensionality of $\mathcal{M}$ implies that the restriction of $\Phi_0$ to $\mathcal{M}$ is
continuous. Further, Lemma~\ref{l2}(1) shows that $\mathcal{M}$ is a two-sided ideal of $L^1(G)$,
and so $h_\lambda\in\mathcal{M}\ast f$ for each $\lambda\in\Lambda$.
Therefore, taking limits on both sides of equation \eqref{1126}
(and using the continuity of $\Phi_0$ on $\mathcal{M}$), we see that
\[
\left\Vert \Phi_0(f) \right\Vert \le \left\Vert P \right\Vert \left\Vert f \right\Vert_{\mathcal{P}_n},
\]
which proves our claim.

Since $\mathcal{T}(G)$ is dense in $\mathcal{P}_n(A)$ (Lemma~\ref{1319}),
it follows that $\Phi_0$ has a continuous extension $\Phi\colon\mathcal{P}_n(A)\to X$.
Take $f\in A$. 
There exists a sequence $(f_k)$ in $\mathcal{T}(G)$ with  $\left\Vert f-f_k \right\Vert_A\to 0$,
so that $P \left(f_k \right)\to P(f)$.
Further, \eqref{1452} gives $\left\Vert f^{*n}-f_k^{*n} \right\Vert_{\mathcal{P}_n}\rightarrow 0$, and consequently
$P \left(f_k \right) =\Phi \left(f_k^{*n} \right)\to\Phi \left( f^{*n} \right)$.
Hence $\Phi \left(f^{*n} \right)=P(f)$.

Finally, we proceed to prove the uniqueness of the map $\Phi$.
Suppose that $\Psi\colon\mathcal{P}_n(A)\to X$ is a continuous linear map
such that $P(f)=\Psi \left(f^{*n} \right)$ for each $f\in A$.
By Theorem~\ref{t1}, $\Psi(f)=\Phi(f)\bigl(=\Phi_0(f)\bigr)$
for each $f\in\mathcal{T}(G)$.
Since $\mathcal{T}(G)$ is dense in $L^1(G)$ (Lemma~\ref{1319}), and both
$\Phi$ and $\Psi$ are continuous, it follows
that $\Psi(f)=\Phi(f)$ for each $f\in A$.
\end{proof}

\begin{example}\label{ex}
Let $G$ be a compact group.
The following convolution algebras satisfy the conditions required in Theorem~\ref{tf}
(see \cite[Remark~38.6]{HR}).
\begin{enumerate}
\item
For $1\le p<\infty$, the algebra $L^p(G)$.
\item
The algebra $C(G)$.
\item
The algebra $A(G)$ consisting of those functions $f\in C(G)$  of the form
\[
f=g\ast h
\]
with $g,h\in L^2(G)$. The norm $\left\Vert \cdot \right\Vert_{A(G)}$ on $A(G)$ is defined by
\[
\left\Vert f \right\Vert_{A(G)}=\inf
\bigl\{
\left\Vert g \right\Vert_2 \left\Vert h \right\Vert_2 : f=g\ast h, \ g,h\in L^2(G)
\bigr\}
\]
for each $f\in A(G)$.
It is worth noting that a function $f\in L^1(G)$ is equal almost everywhere to a function in $A(G)$ if and only if
\[
\sum_{[\pi]\in\widehat{G}}d_\pi\bigl\Vert\widehat{f}(\pi)\bigr\Vert_1 \ < \infty.
\]
Further,
\[
\left\Vert f \right\Vert_{A(G)} = \sum_{[\pi]\in\widehat{G}}d_\pi\bigl\Vert\widehat{f}(\pi)\bigr\Vert_1
\]
for each $f\in A(G)$.
Here $\norm{T}_1$ denotes the trace class norm of the operator $T\in\mathcal{B}(H_\pi)$.
\item
For $1<p<\infty$, the algebra $A_p(G)$ consisting of those functions $f\in C(G)$ of the form
\[
f=\sum_{k=1}^\infty g_k\ast h_k
\]
where $(g_k)$ is a sequence in $L^p(G)$, $(h_k)$ is a sequence in $L^q(G)$  
with $\frac{1}{p}+\frac{1}{q}=1$, and
\[
\sum_{k=1}^\infty \left\Vert g_k \right\Vert_p \left\Vert h_k \right\Vert_q  < \infty.
\]
The norm $\Vert\cdot\Vert_{A_p(G)}$ on $A_p(G)$ is defined by
\[
\left\Vert f \right\Vert_{A_p(G)}=
\inf\biggl\{
\sum_{k=1}^\infty \left\Vert g_k \right\Vert_p \left\Vert h_k \right\Vert_q : 
f=\sum_{k=1}^\infty g_k\ast h_k
\biggr\}
\]
for each $f\in A_p(G)$.  
\item
For $1< p<\infty$, the algebra $S_p(G)$ 
consisting of the functions $f\in L^1(G)$ for which
\[
\sum_{[\pi]\in\widehat{G}} d_\pi\bigl\Vert \widehat{f}(\pi)\bigr\Vert_{S^p(H_\pi)}^p  <  \infty.
\]
Here $\left\Vert T \right\Vert_{S^p(H_\pi)}$ denotes the $p$th Schatten norm of the operator $T\in\mathcal{B}(H_\pi)$. The norm $\left\Vert \cdot \right\Vert_{S_p(G)}$ on $S_p(G)$ is defined by
\[
\left\Vert f \right\Vert_{S_p(G)}=
\left\Vert f \right\Vert_1 +
\left(
\sum_{[\pi]\in\widehat{G}} d_\pi\bigl\Vert \widehat{f}(\pi)\bigr\Vert_{S^p(H_\pi)}^p
\right)^{1/p}
\]
for each $f\in S_p(G)$. 
\end{enumerate}
\end{example}

\begin{remark}
Together with Corollary~\ref{p1}, Theorem~\ref{tf} gives Theorem~\ref{tm}.
\end{remark}

\begin{remark}
It is well-known that the Banach spaces $A(G)$ and $A_p(G)$ of Example~\ref{ex} are particularly important Banach function algebras with respect to the pointwise multiplication. We emphasize that while the references \cite{A1,A2,V,W,WW} apply to the problem of representing the orthogonally additive homogeneous polynomials on the Banach function algebras $A(G)$ and $A_p(G)$ with respect to pointwise multiplication, Theorem~\ref{tf} gives us information about that problem in the case where both $A(G)$ and $A_p(G)$ are regarded as noncommutative Banach algebras with respect to convolution. 
\end{remark}

\begin{remark}
We do not know whether or not the conclusion of Theorem~\ref{tf} must hold for $A=L^\infty(G)$.
\end{remark}

\end{document}